\documentclass[preprint,1p,10pt]{elsarticle}

\journal{arXiv}

\biboptions{sort&compress,comma,round}

\usepackage{graphicx}
\usepackage{amssymb}

\usepackage{bm}

\usepackage{amsmath}

\DeclareSymbolFont{largesymbolsA}{U}{txexa}{m}{n}
\DeclareMathSymbol{\bigtimes}{\mathop}{largesymbolsA}{16}

\usepackage{color}
\usepackage[table]{xcolor}

\usepackage{setspace}

\usepackage{graphics}
\usepackage{amssymb, latexsym, amsmath, epsfig}
\usepackage{graphicx}
 \usepackage{color}
 \usepackage{epstopdf}
 \usepackage{comment}
 \usepackage{soul}
 \usepackage[normalem]{ulem}
  \usepackage{float}

\usepackage{amssymb,amsmath}
\usepackage{amsfonts,amstext,amsthm,amssymb,comment}
\usepackage{epsfig,graphics,color}

\graphicspath{{figures/}}

\newtheorem{theorem}{Theorem}
\newtheorem{lemma}{Lemma}

\numberwithin{equation}{section}

\newtheorem{remark}{Remark}

\newcommand{\T}{\mathcal{T}}
\newcommand{\N}{\mathcal{N}}
\newcommand{\R}{\mathcal{R}}
\newcommand{\I}{\mathcal{I}}
\newcommand{\mo}{\mathcal{O}}

\begin{document}

\begin{frontmatter}


\title{Higher order stable generalized finite element method for the elliptic eigenvalue problem with an interface in 1D}


\author[ad1,ad2]{Quanling Deng\corref{corr}}
\cortext[corr]{Corresponding author}
\ead{Quanling.Deng@curtin.edu.au}

\author[ad1,ad2]{Victor Calo}
\ead{Victor.Calo@curtin.edu.au}
%

%
%
%
%
%
\address[ad1]{Department of Applied Geology, Curtin University, Kent Street, Bentley, Perth, WA 6102, Australia}
\address[ad2]{Mineral Resources, Commonwealth Scientific and Industrial Research Organisation (CSIRO), Kensington, Perth, WA 6152, Australia}
%
%

\begin{abstract}
We study the generalized finite element methods (GFEMs) for the second-order elliptic eigenvalue problem with an interface in 1D. The linear stable generalized finite element methods (SGFEM) were recently developed for the elliptic source problem with interfaces. We first generalize SGFEM to arbitrary order elements and establish the optimal error convergence of the approximate solutions for the elliptic source problem with an interface. We then apply the abstract theory of spectral approximation of compact operators to establish the error estimation for the eigenvalue problem with an interface. The error estimations on eigenpairs strongly depend on the estimation of the discrete solution operator for the source problem.  We verify our theoretical findings in various numerical examples including both source and eigenvalue problems.

\noindent

\end{abstract}


\end{frontmatter}

\vspace{0.5cm}

\section{Introduction}

The finite element method (FEM) is a widely-used and well-understood numerical method for solving partial differential equations arising from science and engineering problems \cite{ciarlet2002finite,brenner2008mathematical,ern2004theory}. The generalized FEM (GFEM) is an extension of the FEM and it is developed to attack problems with non-smooth (low-regularity) solutions, that is, problems involving material discontinuity, moving interfaces, crack propagation, etc. The idea of GFEM is to augment the standard finite element space with a space of non-polynomial functions, called the enrichment space; c.f., \cite{strouboulis2001generalized,strouboulis2000design,strouboulis2000generalized}. The functions in the enrichment space mimic the local behavior of the unknown solution, in particular, the solution behavior near the cracks or the interfaces. In the literature, the method is also known as the extended FEM (XFEM) and both methods are particular instances of the partition of unity method (PUM), which allows the use of any partition of unity combined with local enrichment functions; c.f., \cite{melenk1996partition,melenk1995generalized,babuska1995partition}. We refer further to the review articles \cite{fries2010extended,abdelaziz2008survey,belytschko2009review} and to the references therein for the developments and various applications of GFEMs/XFEMs.

However, the GFEMs suffer from bad conditioning and its conditioning is not robust with respect to the mesh configurations \cite{babuvska2012stable,kergrene2016stable}. In general, the conditioning of GFEM could be much worse than that of the standard FEM, e.g., the condition number of the linear system of GFEM could be of $\mathcal{O}(h^{-4})$ (c.f., \cite{babuvska2012stable}).
The stable GFEM (SGFEM) has been recently introduced to overcome these issues
\cite{babuvska2012stable,kergrene2016stable,babuvska2017strongly}. A GFEM is called stable (SGFEM) if (1) it maintains the optimal order of error convergence, that is, the approximated solution converges to the exact solution at a rate $h^p$ in $H^1$ norm where $p$ is the degree of the underlying FEM polynomial basis functions and $h$ is the discretization parameter (2) the conditioning of GFEM is not worse than that of the standard FEM, that is, the (scaled) condition number of the linear system arising from the GFEM is of the same order $\mathcal{O}(h^{-2})$ as that of a standard FEM in a robust manner with respect to the mesh configurations (c.f., \cite{babuvska2012stable,kergrene2016stable}). 

In literature, most of the work concerning PUM, XFEM, and GFEM focus on linear elements as the optimal convergence requires the lowest regularities on the unknown solutions. We mention the high order XFEM developed in \cite{stazi2003extended,laborde2005high} and high order SGFEM developed in \cite{zhang2014higher}. In particular, the recent work in \cite{zhang2014higher} developed the high order SGFEM and the authors showed that it yields high order convergence both theoretically and analytically. Instead of using the standard Lagrange shape functions of degree $p$ (as for high order FEM), the authors constructed the SGFEM enrichment space based on the piecewise-linear hat-functions (for the linear FEM). It is in a hierarchy manner. The authors used polynomial enrichments and focused on approximating smooth solutions. Optimal error estimates in energy norm were established.

To our best knowledge, most of the work on XFEM/GFEM is for source problems or interface source problems and there is no work on these methods for solving differential eigenvalue problems. Traditionally, differential eigenvalue problems are solved by using standard FEMs (c.f.,\cite{strang1973analysis,bramble1973rate,osborn1975spectral, chatelin1983spectral, babuvska1991eigenvalue, canuto1978eigenvalue,mercier1978eigenvalue,mercier1981eigenvalue}), isogeometric analysis (c.f.,\cite{cottrell2006isogeometric,hughes2008duality,hughes2014finite}), discontinuous Galerkin methods (c.f.,\cite{antonietti2006discontinuous,giani2015hp,gopalakrishnan2015spectral}), etc. We are not going to expand the literature review here and only mention the most-recent quadrature rule blending techniques developed in \cite{ainsworth2010optimally} for FEMs and in \cite{calo2017dispersion, puzyrev2018spectral, deng2017dispersion, deng2018dispersion,calo2017quadrature} for isogeometric analysis. The optimally-blended rules developed in these papers lead to two extra orders of superconvergence for the eigenvalues while maintaining the optimal convergence rates for the eigenfunctions. A generalized Pythagorean eigenvalue error theorem is developed in \cite{puzyrev2017dispersion, bartovn2018generalization} for error analysis. The optimally-blended rules are applied to solve eigenvalue problems for $2n$ order operator in \cite{deng2019optimal} and for Schr$\ddot{\text{o}}$dinger operator in \cite{deng2018isogeometric}.

Another interesting work is the hybrid high-order spectral approximation in \cite{calo2017spectral}, which results in two extra orders of superconvergence for the eigenvalues and one extra order of superconvergence for the eigenfunctions.  However, all these methods were dealing with elliptic differential operators where the diffusion coefficients are continuous. When the elliptic operator contains discontinuous diffusion coefficients (for example, when there is a material discontinuity), all these methods lose the optimal convergence rates on their approximations of both the eigenvalues and the eigenfunctions due to the low regularity of the eigenfunctions. We show will this in the numerical experiments for the standard FEMs.

In this paper, we take a forward step and initiate the study of the recently-developed SGFEM for solving elliptic eigenvalue problem with interfaces. The eigenvalue problem with interfaces could arise from structural vibrations in mechanic engineering where the material discontinuity occurs. The material discontinuity results in a discontinuous diffusion coefficient in the elliptic operator which leads to an interface problem, which brings new challenges to develop robust and stable numerical methods for solving the elliptic eigenvalue problem with interfaces.

We start with developing high order SGFEMs for the 1D interface source problems. The high order generalization is a simple extension of the linear (SGFEM) elements developed in \cite{babuvska2012stable,kergrene2016stable} and we use the same enrichment to capture the local behavior near the interface. We mention that these high order methods are different from the high order SGFEMs developed in \cite{zhang2014higher} where the focus was on approximating smooth solutions. We establish the optimal a priori error estimates where the key is to show the existence of an interpolant in the augmented space which leads to optimal convergence rates. Using the estimates of the discrete solution operators, we then apply the abstract theory of spectral approximation for compact operators to establish the optimal error estimates of SGFEM approximation of the eigenvalue problem with an interface.

The rest of this paper is organized as follows. Section \ref{sec:ps} presents the 1D second-order elliptic eigenvalue problem with an interface. Therein, we derive analytical solutions for several special cases. In order to perform the error analysis, we also introduce the weak solution for the corresponding interface source problem as well as the solution operator. In Section \ref{sec:gfem}, we describe the high order SGFEMs and define the discrete solution operator. In Section \ref{sec:ea}, we first recall the main abstract results of the spectral approximation of compact operators in Hilbert spaces. We then establish the optimal error estimates of arbitrary order SGFEMs for the interface source problem followed by the error analysis of SGFEMs for the eigenvalue problem with an interface. In Section \ref{sec:num}, we present various numerical examples for both the interface source problem and the eigenvalue problem with an interface. Finally, Section \ref{sec:con} presents some concluding remarks and discusses the extensions to multiple dimensions as well as to finite elements with basis functions of high continuities.

\section{Problem statement} \label{sec:ps}
We consider 1D and let $\Omega = (0, 1), \Omega_0 = (0, \gamma), \Omega_1 = (\gamma, 1)$, where $0 < \gamma < 1$. We define $\Gamma := \overline{\Omega}_0 \cap \overline{\Omega}_1$ and call it an interface. The domain $\Omega$ is separated by the interface $\Gamma$ into two sub-domains $\Omega_0$ and $\Omega_1$, where each sub-domain contains one material. We consider the second-order elliptic eigenvalue problem on this type of domain, which we state as: Find the nontrivial eigenpairs $(\lambda, u)$ satisfying

\begin{equation} \label{eq:pde1d}
\begin{aligned}
- ( \kappa u' )' & = \lambda u, \qquad \text{in} \qquad \Omega, \\
u & = 0, \qquad \ \text{on} \qquad \partial\Omega, \\
\end{aligned}
\end{equation}
where the apostrophe $'$ refers to the derivative of the function and  $0 < \tilde \kappa_0 \le \kappa(x) \le \tilde \kappa_1 < + \infty$ is the elliptic coefficient modeling the material properties. The coefficient function $\kappa(x)$ is discontinuous across the interface $\Gamma$ due to the change in material properties. We assume that $\kappa_j (x) := \kappa \big|_{\Omega_j} $ is a continuous function on $\overline{\Omega}_j$ for $j=0,1$.

\subsection{Analytical solutions on special cases}
Analytical eigenpairs of \eqref{eq:pde1d} can be found on the special case where $\kappa_j (x), j =0, 1,$ are constants. Without loss of generality, we assume that $\kappa_0(x) = 1, \kappa_1(x) = \eta.$
We define the piecewise function
\begin{equation}
u(x) = 
\begin{cases}
u_0(x), \qquad 0 \le x \le \gamma, \\
u_1(x), \qquad \gamma \le x \le 1,
\end{cases}
\end{equation}
where both $u_0$ and $u_1$ are defined at $\gamma$ to be the same value.
Due to the discontinuity of $\kappa(x)$ on the interface, the differential equation in \eqref{eq:pde1d} is understood as 
\begin{equation} \label{eq:pde01}
\begin{aligned}
- u''_0 & = \lambda u_0, \qquad 0 < x < \gamma, \\
- \eta u''_1 & = \lambda u_1, \qquad \gamma < x <1, \\
\end{aligned}
\end{equation}
which, by using analytical solution technique of ordinary differential equation, admits a solution of the form
\begin{equation}
u(x) = 
\begin{cases}
u_0 = c_0 \sin(\omega_0 x) + c_1 \cos(\omega_0 x), \qquad 0 \le x \le \gamma, \\
u_1 = d_0 \sin(\omega_1 x) + d_1 \cos(\omega_1 x), \qquad \gamma \le x \le 1,
\end{cases}
\end{equation}
where $c_j, d_j, j =0, 1,$ are constants and 
\begin{equation} \label{eq:w1w2}
\lambda = \omega_0^2 = \eta \omega_1^2.
\end{equation}

Applying the boundary condition $u(0) = 0$ and rewriting $u_1$, we get
\begin{equation}
u(x) = 
\begin{cases}
u_0 = c_0 \sin(\omega_0 x), \qquad 0 \le x \le \gamma, \\
u_1 = \sqrt{d_0^2 + d_1^2} \sin(\omega_1 x + \theta), \quad \cos(\theta) = \frac{d_0}{\sqrt{d_0^2 + d_1^2}}, \qquad \gamma \le x \le 1,
\end{cases}
\end{equation}
which, by introducing new constants $c,d$ and applying the boundary condition $u(1) = 0$, is further rewritten as 
\begin{equation} \label{eq:u1dsol}
u(x) = 
\begin{cases}
u_0 = c \sin(\omega_0 x), \qquad 0 \le x \le \gamma, \\
u_1 = d \sin(\omega_1 x - \omega_1), \qquad \gamma \le x \le 1.
\end{cases}
\end{equation}

We are interested in seeking the solutions such that 
\begin{equation} \label{eq:condgm}
u_0(\gamma) = u_1(\gamma), \qquad -u'_0(\gamma) = -\eta u'_1(\gamma),
\end{equation}
where these two conditions mean that the solution and its flux are continuous at $\gamma$. However, there are four unknowns in \eqref{eq:u1dsol} and we have three equations in \eqref{eq:condgm} and \eqref{eq:w1w2}. The last condition comes from a normalization of $u$ on $\Omega$. To obtain a general solution, without loss of generality, one can assume $c=1$.

Denoting $\rho = \sqrt{\eta}$, from \eqref{eq:w1w2}, we have $\omega_0 = \rho \omega_1$. Applying the conditions \eqref{eq:condgm}, we obtain

\begin{equation} \label{eq:condgm1}
\begin{aligned}
\sin(\rho \omega_1 \gamma) & = d \sin(\omega_1 \gamma - \omega_1), \\
\cos( \rho \omega_1 \gamma) & = d \rho \omega_1 \cos(\omega_1 \gamma - \omega_1),
\end{aligned}
\end{equation}
where $\rho, \gamma$ are data and $\omega_1, d$ are unknown. For any $\eta$ and $\gamma$, the 1D eigenvalue problem \eqref{eq:pde1d} has eigenvalues \eqref{eq:w1w2} and eigenfunctions \eqref{eq:u1dsol} once $\omega_1$ and $d$ are solved from \eqref{eq:condgm1}. We present the analytic solutions for the following special cases.

\subsubsection{Special case 1: $\eta =1$}
Given $\eta=1$, there is no discontinuity in $\kappa$ and \eqref{eq:pde1d} is the classic 1D Laplacian eigenvalue problem.  For arbitrary $0<\gamma<1$, the exact eigenpair is $(\lambda = (n\pi)^2, u=\sin(n\pi x)), n=1,2,\cdots$.

\subsubsection{Special case 2: $\eta =4, \gamma = 1/3$} \label{sec:c2}
Solving \eqref{eq:condgm1}  yields the solutions
\begin{equation} 
\begin{aligned}
d & = -1,  \quad & \omega_1  &= 3n\pi - \frac{3 \pi}{4}, n=1,2,\cdots, \\
d & = -1,  \quad & \omega_1  &= 3n\pi + \frac{3 \pi}{4}, n=0,1,2,\cdots,\\
d & = \frac{1}{2},  \quad & \omega_1  &= 3n\pi, n=1,2,\cdots, \\
d & = \frac{1}{2},  \quad & \omega_1  &= 3n\pi + \frac{3 \pi}{2}, n=0,1,2,\cdots.\\
\end{aligned}
\end{equation}
The exact eigenvalues are then listed in an increasing order as
\begin{equation}
\begin{aligned}
\lambda = & \frac{9 }{4} n^2 \pi^2, \qquad n=1,2, \cdots.
\end{aligned}
\end{equation}

The eigenfunctions are given in \eqref{eq:u1dsol} with $c=1$ and they can be further normalized. Moreover, $u_n$ are $C^1$ for $n = 1,3,5, \cdots$ while they are $C^0$ for $n=2,4, \cdots.$ We obtain similar results for the cases with other rational constants $\eta$ and $\gamma$.

\subsubsection{Special case 3: transcendental numbers} \label{sec:c3}
For cases involving transcendental numbers, such as $\pi, e$, \textbf{Mathematica} is unable to give all exact or numerical solutions. However, we can approximate them numerically by fixing an interval on eigenvalues. For example, set $\gamma = 1/\pi, \eta = e^2,$ then the first six solutions (with 18 decimal digits) are
\begin{equation} \label{eq:refsol4}
\begin{aligned}
d & = -0.964943954172912233,  \quad & \omega_1 = 2.14610955883566090, \\
d & = 0.412394727933596716,  \quad & \omega_1 = 3.86404207192115656, \\
d & = -0.742060865115781954,  \quad & \omega_1 = 6.37668915886029808, \\
d & = 0.565172936488676170,  \quad & \omega_1 = 7.81591521600899775, \\
d & = -0.508613379389131093,  \quad & \omega_1 = 10.4448447099813517, \\
d & = 0.817433181806672682,  \quad & \omega_1 = 11.9357905821583099. \\
\end{aligned}
\end{equation}

In the numerical experiment section, we use these numerical solutions as a reference solutions to the exact ones.

\subsection{Weak solution}
Let $H^1(\Omega)$ be the standard Hilbert space of functions and $H^1_0(\Omega)$ be the Hilbert space with functions vanishing at the boundary. 
In weak formulation, the problem \eqref{eq:pde1d} reads as follows: Find $(\lambda, u)\in \mathbb{R}_{>0} \times H^1_0(\Omega)$ such that 
\begin{equation} \label{eq:vf}
a(u, w) = \lambda b(u, w), \qquad \forall \ w \in H^1_0(\Omega),
\end{equation}
with the bilinear forms $a$ and $b$ defined on $H^1_0(\Omega)\times H^1_0(\Omega)$ and $L^2(\Omega)\times L^2(\Omega)$ as
\begin{equation}
a(v, w) = (\kappa v', w')_{L^2(\Omega)}, \qquad  b(v, w) = (v, w)_{L^2(\Omega)},
\end{equation}
where $(\cdot, \cdot)_{L^2(\Omega)}$ denotes the inner product in $L^2(\Omega)$.

\subsection{Source problem and solution operator} \label{sec:T}
The source problem associated with the eigenvalue problem~\eqref{eq:pde1d} is: For $f \in L^2(\Omega)$, find $u$ such that 
\begin{equation} \label{eq:pdef}
\begin{aligned}
- (\kappa u')' & = f, \qquad \text{in} \qquad \Omega, \\
u & = 0, \qquad \text{on} \qquad \partial\Omega, \\
\end{aligned}
\end{equation}
while the corresponding weak formulation is as follows: For $f \in L^2(\Omega)$, find $u\in H^1_0(\Omega)$ such that 
\begin{equation} \label{eq:weakf}
a(u,w) = b(f, w), \qquad \forall \ w \in H^1_0(\Omega).
\end{equation}

The solution operator associated with \eqref{eq:weakf} is 
$T: L^2(\Omega) \rightarrow L^2(\Omega)$, so that we have $T(f)\in H^1_0(\Omega) \subset L^2(\Omega)$ and
\begin{equation} \label{eq:weakTf}
a(T(f),w) = b(f, w), \qquad \forall \ w \in H^1_0(\Omega).
\end{equation}
 
By the Rellich--Kondrachov Theorem (see, e.g., \cite[Thm.~1.4.3.2]{grisvard1985elliptic}), 
$T$ is compact from $L^2(\Omega)$ to $L^2(\Omega)$. The reason for introducing the solution operator $T$ is that  $(\lambda, u)\in\mathbb{R}_{>0}\times H^1_0(\Omega)$ is an eigenpair for \eqref{eq:vf} if and only if $(\mu,u)\in\mathbb{R}_{>0}\times H^1_0(\Omega)$ with $\mu = \lambda^{-1}$ is an eigenpair of $T$. 
We also define the adjoint solution operator $T^*: L^2(\Omega) \rightarrow L^2(\Omega)$ such that, for all $g \in L^2(\Omega)$, $T^*(g)\in H^1_0(\Omega)$ and 
\begin{equation} \label{eq:weakT*}
a(w,T^*(g)) = b(w,g), \qquad \forall \ w \in H^1_0(\Omega).
\end{equation}
The symmetry of the bilinear forms $a$ and $b$ implies that $T=T^*$; however, allowing more generality, we keep a distinct notation for the two operators. Since, in general, we have 
\begin{equation} \label{eq:tts}
(T(f),g)_{L^2(\Omega)} = a(T(f),T^*(g)) = (f,T^*(g))_{L^2(\Omega)},
\end{equation}
we infer that $T^*$ is the adjoint operator of $T$, once the duality product is identified with the inner product in $L^2(\Omega)$. Therefore, in the present symmetric context, the operator $T$ is self-adjoint.

\section{Stable generalized finite element methods (SGFEMs)} \label{sec:gfem}

The main idea of GFEM is to enlarge the standard FEM space with an enrichment space consisting of certain enrichment functions. The GFEM framework for solving \eqref{eq:pde1d} is: Find $(\lambda^h, u^h)\in \mathbb{R} \times V^h$ such that 
\begin{equation} \label{eq:vfh}
a(u^h, w^h) = \lambda^h b(u^h, w^h), \qquad \forall \ w^h \in V^h,
\end{equation}
where the approximation space $V^h \subset H_0^1(\Omega)$ and is given by
\begin{equation}
V^h = V^h_{FEM} \oplus V^h_{ENR} = \{ w=w_1 + w_2: w_1 \in V^h_{FEM}, w_2 \in V^h_{ENR} \},
\end{equation}
with $V^h_{FEM}$ being the standard FEM space and $V^h_{ENR}$ being the enrichment space. We specify these spaces for 1D problems with an interface as follows.

Let $\T_h$ be a partition of the unit interval $\overline{\Omega} = [0,1]$ with nodes $0 = x_0 < x_1 < \cdots < x_N = 1.$ We define the sets $\N^h := \{0, 1, 2,\cdots, N \}$ and $\N^h_e = \{ 1, 2,\cdots, N \}$. Let $\tau_j = [x_{j-1}, x_j]$ be an element with size $h_j = x_j - x_{j-1}$ for $j\in \N^h_e$.  There are $N$ elements. Let $h = \max_{j\in \N^h_e} h_j$. Since $0 < \gamma <1$, there exists $r \in \N^h_e$ such that $\gamma \in \tau_r$. There is at least one such $r$; there are two such $r$ when $\gamma$ coincides with a node in the partition $\T_h$. In the later case, the partition is called a fitting mesh, in which the enrichment is not necessary (c.f., \cite{kergrene2016stable}). In this paper, we consider non-fitting meshes, that is, we assume $\gamma$ is not located at a node. In this case, we denote the element $\tau_{r-1/2} = [x_{r-1}, \gamma]$ and $\tau_{r+1/2} = [\gamma, x_r]$.

Let $p$ be the degree of a polynomial and define a set $\N^h_0 = \{1, 2,\cdots, pN - 1 \}$. 
Let $N_j^p, j \in \N^h_0,$ be the $p$-th order standard $C^0$ B-spline basis functions (which is equivalent to the Lagrange basis functions). In this setting, we specify the FEM space and the enrichment space as
\begin{equation}
V^h_{FEM} = \text{span} \{N_j^p: j  \in \N^h_0\}, \qquad V^h_{ENR} = \text{span} \{w N_j^p: j  \in \R^h \subset \N^h_0\},
\end{equation}
where $w$ is the enrichment and $V^h_{ENR}$ is called the enrichment space of GFEM. The function $w$ is chosen such that it mimics the exact solution near the interface. The set $\{ x_j \}_{j\in \R^h}$ denotes the set of enriched nodes. For $w=0$, we remove the enrichment space and the method reduces to the standard FEM.

The GFEM eigenfunction $u^h \in V^h$ is obtained in the form
\begin{equation} \label{eq:uh}
u^h = \sum_{j\in \N^h_0} U^F_j N_j^p + \sum_{j\in \R^h} U^E_j w N_j^p 
\end{equation}
by solving the generalized matrix eigenvalue problem
\begin{equation} \label{eq:mevp}
\begin{bmatrix}
\mathbf{K}_{FF} & \mathbf{K}_{FE} \\
\mathbf{K}_{EF} & \mathbf{K}_{EE} \\
\end{bmatrix}
\begin{bmatrix}
\mathbf{U}^F \\
\mathbf{U}^E \\
\end{bmatrix}
= \lambda^h
\begin{bmatrix}
\mathbf{M}_{FF} & \mathbf{M}_{FE} \\
\mathbf{M}_{EF} & \mathbf{M}_{EE} \\
\end{bmatrix}
\begin{bmatrix}
\mathbf{U}^F \\
\mathbf{U}^E \\
\end{bmatrix},
\end{equation}
where the block matrices are defined (using the bilinear forms) with entries
\begin{equation}
\begin{aligned}
(\mathbf{K}_{FF})_{jk} & = a(N_k^p, N_j^p), \quad & (\mathbf{M}_{FF})_{jk} & = b(N_k^p, N_j^p), \\
(\mathbf{K}_{EE})_{jk} & = a(wN_k^p, wN_j^p), \quad & (\mathbf{M}_{EE})_{jk} & = b(wN_k^p, wN_j^p), \\
(\mathbf{K}_{FE})_{jk} & =  (\mathbf{K}_{EF})_{kj} = a(wN_k^p, N_j^p), \quad & (\mathbf{M}_{FE})_{jk} & =  (\mathbf{M}_{EF})_{kj} = b(wN_k^p, N_j^p). \\
\end{aligned}
\end{equation}
Herein, the matrix eigenvalue problem \eqref{eq:mevp} arises from \eqref{eq:vfh}. $\lambda^h$ is GFEM approximated eigenvalue. $F$ corresponds to the standard  FEM part while $E$ corresponds to the enrichment part. $\mathbf{U}^S$ is an unknown vector consists of $U_j^S$ for $S = F, E$ as in the solution representation \eqref{eq:uh}. Since both the stiffness and mass matrices are positive-definite, the generalized matrix eigenvalue problem \eqref{eq:mevp} has a unique set of eigenpairs.

Different enrichment choices for $w$ and set $\R^h$ lead to different GFEMs, such as the Geometric GFEM, Topological GFEM, M-GFEM, and SGFEM; see \cite{kergrene2016stable,babuvska2017strongly} for details.
 In general, the conditioning of GFEM is significantly worse than that of FEM, which results in extreme difficulties when solving the corresponding linear system. The conditioning numbers of standard FEM, Geometric GFEM,  SGFEM, is $\mo(h^{-2})$ , $\mo(h^{-4})$, $\mo(h^{-2})$, respectively. The conditioning of M-GFEM is not robust in higher dimensions (2D or 3D). The energy error estimates for the standard FEM, Topological GFEM, Geometric GFEM,  M-GFEM, and SGFEM, are $\mo(h^{1/2}), \mo(h^{1/2}), \mo(h), \mo(h), \mo(h)$, respectively. Among them, the SGFEM is accurate as well as robustly well-conditioned. The SGFEM with linear elements is well-studied in \cite{kergrene2016stable,babuvska2017strongly}, thus, we focus on its generalization to arbitrary orders.


\subsection{Stable GFEM (SGFEM)} \label{sec:sgfem}
Let $w^* = |x - \gamma|$. The enrichment is defined as a continuous and locally-supported function, that is $w= \I_h w^* - w^*$ (in \cite{kergrene2016stable}, it is defined as $w= w^* - \I_h w^*$, we define in this way so that it is positive over the element $\tau_r$ but they are essentially the same), where $\I_h w^*$ is the piecewise linear interpolant of $w^*$ with respect to the mesh $\T_h.$ The support of the enrichment is on the element $\tau_r = [x_{r-1}, x_{r}]$. For $p$-th order basis function, there are $p+1$ enrichment  functions, thus the dimension of $\mathbf{M}_{EE}$ or $\mathbf{K}_{EE}$ is $p+1$.

\subsection{Discrete solution operator} \label{sec:Th}
The corresponding GFEM discretization of the source problem \eqref{eq:pdef} reads as follows: Find $u^h \in V^h$ such that 
\begin{equation} \label{eq:vfhf}
a(u^h, w^h) = b(f, w^h), \qquad \forall \ w^h \in V^h.
\end{equation}
We define the GFEM solution operator $T_h: L^2(\Omega) \to V^h$ so that
\begin{equation} \label{eq:Th}
a(T_h(f), w^h) =  b(f, w^h), \qquad \forall \ w^h \in V^h.
\end{equation}

As for the continuous solution operator, we also define the discrete adjoint solution operator $T^*_h: L^2(\Omega) \rightarrow V^h$ such that, for all $g \in L^2(\Omega)$, 
\begin{equation} \label{eq:T*h}
a(w^h,T^*_h(g)) = b(w^h, g), \qquad \forall \ w^h \in V^h.
\end{equation}
Similarly, the symmetry of the bilinear forms $a$ and $b$ implies that $T_h =T^*_h.$

\section{Error analysis} \label{sec:ea}

In this section, we first establish error estimates for the SGFEM approximations for the source problem and then for the eigenvalue problem. In what follows, we use the symbol $C$ to denote a generic constant (its value can change at each occurrence) that can depend on the mesh regularity, the polynomial degree $p$ and the domain $\Omega$, but is independent of the mesh size $h$.

\subsection{Spectral approximation theory for compact operators} 
\label{sec:theory}

Let us now briefly recall the main results we use concerning the spectral approximation of compact operators in Hilbert spaces. Let $L$ be a Hilbert space with inner product denoted by $(\cdot,\cdot)_L$, and let $T \in \mathcal{L}(L;L)$; assume that $T$ is compact. We do not assume for the abstract theory that $T$ is self-adjoint and we let $T^*\in\mathcal{L}(L;L)$ denote the adjoint operator of $T$. Let $T_n\in \mathcal{L}(L;L)$ be a member of a sequence of compact operators that converges to $T$ in operator norm, i.e., 
\begin{equation} \label{eq:normc}
\lim_{n \to +\infty} \|  T -  T_n \|_{\mathcal{L}(L; L)} = 0,
\end{equation}
and let $T_n^*\in \mathcal{L}(L;L)$ be the adjoint operator of $T_n$.
We want to study how well the eigenvalues and the eigenfunctions of $T_n$ approximate those of $T$. 
Let $\sigma(T)$ denote the spectrum of the operator $T$ and let $\mu \in \sigma(T)\setminus\{0\}$ be a nonzero eigenvalue of $T$. Let $\varsigma$ be the ascent of $\mu$, i.e., the smallest integer $\varsigma$ such that 
$\text{ker}(\mu I -  T)^\varsigma = \text{ker}(\mu I -  T)^{\varsigma+1}$,
where $I$ is the identity operator. Let also 
\begin{equation} \label{eq:def_Gmu}
G_\mu = \text{ker}(\mu I-T)^\varsigma, \qquad G^*_\mu = \text{ker}(\mu I-T^*)^\varsigma,
\end{equation} 
and $\varrho= \text{dim}(G_\mu)$ (this integer is called the algebraic multiplicity of $\mu$; note that $ \varrho \ge \varsigma$). 

\begin{theorem}[Convergence of the eigenvalues] \label{thm:eve0}
Let $\mu\in \sigma(T)\setminus\{0\}$. 
Let $\varsigma$ be the ascent of $\mu$ and let $\varrho$ be its algebraic multiplicity.
Then there are $\varrho$ eigenvalues of $T_n$, denoted as $\mu_{n,1}, \cdots, \mu_{n,\varrho}$, that converge to $\mu$ as $n \to +\infty$. Moreover, letting $\langle \mu_n \rangle = \frac{1}{\varrho} \sum_{j=1}^\varrho\mu_{n,j}$ denote their arithmetic mean, 
there is $C$, depending on $\mu$ but independent of $n$, such that 
\begin{equation} \label{eq:cv_eigenval}
\begin{aligned}
\max_{1 \le j \le\varrho } |\mu - \mu_{n,j}|^\varsigma & + |\mu - \langle \mu_n \rangle | \le C \Big( \sup_{\substack{0\ne \phi \in G_\mu\\ 0\ne \psi \in G_\mu^*}} \frac{|((T-T_n)\phi, \psi)_L|}{\| \phi \|_L  \| \psi \|_L } \\
& + \|(T-T_n)|_{G_\mu} \|_{\mathcal{L}(G_\mu;L)}  \|(T-T_n)^*|_{G_\mu^*} \|_{\mathcal{L}(G_\mu^*;L)} \Big).
\end{aligned}
\end{equation}
\end{theorem}

\begin{remark}[Convergence of the arithmetic mean]
Equation \eqref{eq:cv_eigenval} shows that for $\varsigma\ge 2$, the arithmetic mean of the eigenvalues has a better convergence rate than each eigenvalue individually. 
\end{remark}

\begin{theorem}[Convergence of the eigenfunctions] \label{thm:efe0}
Let $\mu\in \sigma(T)\setminus\{0\}$ with  
ascent $\varsigma$ and algebraic multiplicity $\varrho$.
Let $\mu_{n, j}$ be an eigenvalue of $T_n$ that converges to $\mu$. Let $w_{n,j}$ be a unit vector in $\text{ker}(\mu_{n,j}I-T_n)^\ell$ for some positive integer $\ell \le \varsigma$. Then, for any integer $\varrho$ with $\ell \le \varrho\le \varsigma$, there is a vector $u_\varrho\in \text{\em ker}(\mu I- T)^\varrho\subset G_\mu$ such that
\begin{equation}
\|u_\varrho - w_{n,j}\|_L \le C \|(T-T_n)|_{G_\mu}\|_{\mathcal{L}(G_\mu; L)}^{\frac{\varrho - \ell + 1}{\varsigma}},
\end{equation}
where $C$ depends on $\mu$ but is independent of $n$.
\end{theorem}

\subsection{Error estimation for the source problem}
In this section, we generalize the analysis results established in \cite{babuvska2012stable,babuvska2017strongly} for the linear SGFEM discretization of the source problem to arbitrary order in 1D setting. We use the standard notations for the norms and semi-norms of the Sobolev spaces. Let $S\subset \overline\Omega$ be a subdomain. We denote by $| \cdot |_{H^1(S)}$,  $| \cdot |_{H^2(S)}$, and $\| \cdot \|_{L^2(S)}$ the $H^1$ semi-norm, $H^2$ semi-norm, and $L^2$ norm. Let $\| \cdot \|_E = \sqrt{a(\cdot, \cdot)}$ be the energy norm, which is equivalent to the $H^1$ semi-norm.

First of all, the bilinear form is coercive and bounded \cite{babuvska2017strongly}, i.e., there holds
\begin{equation} \label{eq:a}
a(v, v) \ge \alpha | v |^2_{H^1(\Omega)} \qquad \text{and} \qquad a(v, w) \le \beta | v |_{H^1(\Omega)} | w |_{H^1(\Omega)},
\end{equation}
where $\alpha, \beta$ are constants mainly depending on the coefficient $\kappa$. We first present the following Lemma which plays a crucial role in proving the optimal convergence for arbitrary order $p$. For $p=1$, we refer to \cite{babuvska2017strongly}.

\begin{lemma}[Local interpolant] \label{lem:P}
Let $P(x)$ be a piecewise polynomial on a reference interval $[0, 1]$ with an interface $0<\nu<1$, which is defined as
\begin{equation}
P(x) = 
\begin{cases}
P_0(x) = a_p x^p + \cdots + a_1 x + a_0, \  0 \le x \le \nu, \\
P_1(x) = b_p x^p + \cdots + b_1 x + b_0, \  \nu \le x \le 1, \\
\end{cases}
\end{equation}
where both $P_0$ and $P_1$ are defined at $\nu$ to be equal such that $P(x)$ is continuous. Let $\tilde N_j(x) = x^j, j=0,1,\cdots, p$, be defined on $[0, 1]$. Then there exists a unique set of coefficients $\alpha_j$ and $\beta_j$ with $j=0,1, \cdots, p$ such that
\begin{equation} \label{eq:qp0}
P(x) = \sum_{j=0}^p \big( \alpha_j \tilde N_j + \beta_j w \tilde N_j \big) := Q(x), \quad \forall \ x  \in [0, 1],
\end{equation}
where $w$ is defined on $[0,1 ]$ as in subsection \ref{sec:sgfem} for SGFEM.
\end{lemma}

\begin{proof}
Firstly, since $P(x)$ is continuous at $x=\nu$, there holds
\begin{equation} \label{eq:b0}
b_0  = \sum_{l=1}^p  (b_l -a_l) \nu^l - a_0,
\end{equation}
thus, $b_0$ is a parameter determined by other coefficients. On the other hand, $w$ is written as
\begin{equation}
w(x) = 
\begin{cases}
w_0(x) = 2(1 - \nu) x, \  0 \le x \le \nu, \\
w_1(x) = 2\nu (1 - x), \  \nu \le x \le 1. \\
\end{cases}
\end{equation}

We first show the existence. Consider the $n$-th order derivatives of $P(x)$ and $Q(x)$ at $x=0$ (right derivatives) and $x=1$ (left derivatives) for $n=1,2, \cdots, p+1$ and $P(x) = Q(x)$ for $x=0, \nu$. We set and solve
\begin{equation}
\begin{aligned}
Q^{(n)}(x) & = P^{(n)}(x) \quad \text{for} \quad x=0,1, \qquad \forall \ n=1,2, \cdots, p+1, \\
P(0) & = Q(0), \qquad P(\nu)  = Q(\nu),
\end{aligned}
\end{equation}
which is
\begin{equation} \label{eq:qp1}
\begin{aligned}
\alpha_0 & = a_0, \\
\alpha_k + 2(1 - \nu) \beta_{k-1} & = a_k, \ k = 1,2, \cdots, p, \\
(p+1)! \cdot 2(1- \nu) \beta_p & = 0
\end{aligned}
\end{equation}
 and
\begin{equation} \label{eq:qp2}
\begin{aligned}
\sum_{l=0}^p \nu^l \alpha_l + 2 (1 - \nu) \nu^{l+1} \beta_l & =  \sum_{l=0}^p a_l \nu^l, \\
\sum_{l=k}^p A_l^k \alpha_k - 2\nu \Big( A_k^k \beta_{k-1} + \sum_{l=k}^p (A_{l+1}^k - A_l^k) \beta_l \Big)& = \sum_{l=k}^p A_l^k b_k, \ k = 1,2, \cdots, p, \\
(p+1)! \cdot (-2\nu) \beta_p & = 0
\end{aligned}
\end{equation}
with $A_l^k = \frac{l!}{(l-k)!}.$
 Solving \eqref{eq:qp1} for $\alpha_l, l=0,1, \cdots, p$ and $\beta_p (=0)$ to plug into \eqref{eq:qp2} yields a reduced system
\begin{equation} 
\begin{aligned}
2 k! \beta_{k-1} + 2 \sum_{l=k+1}^p (A_l^k - \nu A_{l-1}^k) \beta_{l-1} & = \sum_{l=k}^p A_l^k a_l - \tilde b_k, \ k = 1,2, \cdots, p, \\
\end{aligned}
\end{equation}
which is written as 
\begin{equation*} 
\begin{bmatrix}
1 & 2-\nu & 3-2\nu & \cdots & \cdots & \cdots \\
 0 & 2! & 2!(3-\nu) & \cdots & \cdots & \cdots\\
0 & 0 & 3! & \cdots & \cdots & \cdots\\
\vdots & \vdots & \vdots & \ddots & \cdots & \cdots\\
0 & 0 & 0 & \cdots & (p-1)! & (p-1)! (p - \nu) \\
0 & 0 & 0 & \cdots & 0 & p! \\
\end{bmatrix}
\begin{bmatrix}
\beta_0 \\
\beta_1 \\
\beta_2 \\
\vdots \\
\beta_{p-2} \\
\beta_{p-1} \\
\end{bmatrix}
=
\begin{bmatrix}
\frac{\sum_{k=1}^p A_k^1 (a_k - b_k) }{2} \\
\frac{\sum_{k=2}^p A_k^2 (a_k - b_k) }{2} \\
\frac{\sum_{k=3}^p A_k^3 (a_k - b_k) }{2} \\
\vdots \\
\frac{A_p^p (a_p - b_p) }{2} \\
\end{bmatrix}.
\end{equation*}
Obviously, there is a unique solution to the above matrix problem and hence a unique set of $\alpha_l, \beta_l, l=0,1,\cdots,p$. This completes the proof.
\end{proof}

\begin{lemma}[Existence of global interpolant] \label{lem:uh*}
Let $u \in C^0(\overline\Omega)$ be the solution of \eqref{eq:pdef} and assume that $u|_{\Omega_j} \in H^{p+1}(\Omega_j)$ for $j=0,1$. For a fixed grid with $N$ elements, let $V^h$ be the $p$-th order SGFEM approximation space. Then there exists $u^*_h \in V^h$ such that
\begin{equation} \label{eq:ie}
\begin{aligned}
\| u - u^*_h \|_{L^2(\Omega)} & \le Ch^{p+1}(|u |_{H^{p+1}(\Omega_0)} + |u |_{H^{p+1}(\Omega_1)}),  \\
| u - u^*_h |_{H^1(\Omega)} & \le Ch^p(|u |_{H^{p+1}(\Omega_0)} + |u |_{H^{p+1}(\Omega_1)}),
\end{aligned}
\end{equation}
where $C$ is independent of $h$.  
\end{lemma}

\begin{proof}
Let $u^{\#}_{h,j}$ defined on $\overline\Omega_j$ be the piecewise $p$-th order Lagrange polynomial interpolant of $u$ restricted to $\overline\Omega_j$ for $j=0,1$. By piecewise we mean that it is a $p$-th order polynomial when restricted to the elements $\tau_l, l=1,2,\cdots, r-1$ and $\tau_{j+1/2}$ (that is all the elements in $\Omega_0$) for $j=0$ and $\tau_l, l=r+1,r+2,\cdots, N$ and $\tau_{j+1/2}$ (that is all the elements in $\Omega_1$) for $j=1$. Since $u|_{\Omega_j} \in H^{p+1}(\Omega_j)$, from interpolation theory (see for example\cite{ern2004theory,brenner2008mathematical,johnson2009numerical}), there holds
\begin{equation} \label{eq:l2h1uh01}
\| u - u^{\#}_{h,j} \|_{L^2(\Omega_j)} \le Ch^{p+1}|u |_{H^{p+1}(\Omega_j)}  \quad \text{and} \quad | u - u^{\#}_{h,j} |_{H^1(\Omega_j)} \le Ch^p|u |_{H^{p+1}(\Omega_j)}.
\end{equation}

Now, we define a piecewise polynomial
\begin{equation}
u^{\#}_h =
\begin{cases}
  u^{\#}_{h,0}, \quad \text{for} \quad x \in \overline\Omega_0 = [0, \gamma], \\
  u^{\#}_{h,1}, \quad \text{for} \quad x \in \overline\Omega_1 = [\gamma, 1], \\
\end{cases}  
\end{equation}
where both $u^{\#}_{h,0}$ and $u^{\#}_{h,1}$ are defined at $x=\gamma$ as $u^{\#}_{h,0} = u^{\#}_{h,1} = u(\gamma)$ since $u \in C^0(\overline\Omega)$. Using \eqref{eq:l2h1uh01}, we calculate
\begin{equation}
\| u - u^{\#}_h \|^2_{L^2(\Omega)} = \| u - u^{\#}_{h,0} \|^2_{L^2(\Omega_0)} + \| u - u^{\#}_{h,1} \|^2_{L^2(\Omega_1)} \le Ch^{2p+2} (|u |^2_{H^{p+1}(\Omega_0)} + |u |^2_{H^{p+1}(\Omega_0)}),
\end{equation}
which, by taking the square root on both sides, we obtain
\begin{equation}
\| u - u^{\#}_h \|_{L^2(\Omega)} \le Ch^{p+1} (|u |_{H^{p+1}(\Omega_0)} + |u |_{H^{p+1}(\Omega_0)}).
\end{equation}

Similarly, one obtains
\begin{equation}
  | u - u^{\#}_h |_{H^1(\Omega)} \le Ch^p(|u |_{H^{p+1}(\Omega_0)} + |u |_{H^{p+1}(\Omega_1)}).
\end{equation}

Now we show that $u^{\#}_h \in V^h$. Let $u^o_h$ be the piecewise $p$-th order Lagrange polynomial interpolant of $u$ on the mesh $\T_h$. Then, $u^o_h \in V^h_{FEM} \subset V^h$. To show that $u^{\#}_h \in V^h$, it is equivalent to show that $u^{\#}_h - u^o_h \in V^h$.
For SGFEM, we note that $u^{\#}_h - u^o_h$ is continuous and has the support over only one element $\tau_r = [x_{r-1}, x_r]$ since both interpolant $u^{\#}_h$ and $u^o_h$ are the same outside the element $\tau_r$. Moreover, $u^{\#}_h - u^o_h $ is a $p$-th order polynomial when restricted to $[x_{r-1}, \gamma]$ or $[\gamma, x_r]$. With a linear transformation from $[x_{r-1}, \gamma]$ to $[1, \nu]$ and $[\gamma, x_r]$ to $[\nu, 1]$, using the Lemma \ref{lem:P} we conclude that $u^{\#}_h - u^o_h \in V^h$. Thus, setting $u^*_h = u^{\#}_h$ completes the proof.
\end{proof}

\begin{remark}
$u^*_h$ in Lemma \ref{lem:uh*} is not unique. The piecewise interpolant $u^{\#}_h$ defined above  is such a function satisfying \eqref{eq:ie}. 
\end{remark}

Now we present the optimal error estimations for the source problem for arbitrary order SGFEM in 1D.

\begin{theorem}[Error estimates for source problem] \label{thm:ef}
Let $u \in C^0(\overline\Omega)$ be the solution of \eqref{eq:pdef} and assume that $u|_{\Omega_j} \in H^{p+1}(\Omega_j)$ for $j=0,1$. Assuming that $u^h \in V^h$ is the solution of the source problem \eqref{eq:vfhf} where $V^h$ is the $p$-th order SGFEM approximation space, then there holds
\begin{equation}
\begin{aligned}
\| u - u^h \|_{L^2(\Omega)} & \le Ch^{p+1}(|u |_{H^{p+1}(\Omega_0)} + |u |_{H^{p+1}(\Omega_1)}),  \\
| u - u^h |_{H^1(\Omega)} & \le Ch^p(|u |_{H^{p+1}(\Omega_0)} + |u |_{H^{p+1}(\Omega_1)}),
\end{aligned}
\end{equation}
where $C$ is independent of $h$.  
\end{theorem}

\begin{proof}
Using the coercivity and boundedness of the bilinear form $a(\cdot, \cdot)$ in \eqref{eq:a} and the Galerkin orthogonality (from \eqref{eq:weakf} and \eqref{eq:vfhf}), the Cea's  Lemma yields
\begin{equation}
| u - u^h |_{H^1(\Omega)} \le \frac{\beta}{\alpha}  | u - v |_{H^1(\Omega)}, \quad \forall \ v\in V^h.
\end{equation}

Applying Lemma \ref{lem:uh*} with the (discrete) Poincar{\'e}--Friedrichs inequality, (see for example \cite{adams2003sobolev,suli2012lecture}) completes the proof.
\end{proof}

\subsection{Error estimation for the eigenvalue problem}

The goal of this section is to perform the error analysis of the discrete eigenvalue problem~\eqref{eq:vfh} by using the abstract theory outlined in Section~\ref{sec:theory} in the Hilbert space $L=L^2(\Omega)$. Let $T,T^* : L^2(\Omega) \rightarrow H^1_0(\Omega)\subset L^2(\Omega)$ be the exact solution and adjoint solution operators defined in Section~\ref{sec:T}. Let $T_h, T_h^*: L^2(\Omega) \rightarrow V^h \subset H^1_0(\Omega)\subset L^2(\Omega)$ be the SGFEM discrete solution and adjoint solution operators defined in \ref{sec:Th}.
In order to apply the abstract theory from Section \ref{sec:theory}, we first quantify the smoothness of the functions in the subspaces $G_\mu$ and 
$G_\mu^*$ defined in~\eqref{eq:def_Gmu} and assume that there is a constant $C$ so that 
\begin{equation} \label{eq:smT}
\begin{alignedat}{2}
\|\phi\|_{H^{p+1}(\Omega)} + \|T(\phi)\|_{H^{p+1}(\Omega)} & \le C \|\phi\|_L, &\qquad& \forall \phi \in G_\mu \subset H^{p+1}(\Omega), \\
\|\psi\|_{H^{p+1}(\Omega)} + \|T^*(\psi)\|_{H^{p+1}(\Omega)} & \le C \|\psi\|_L, &\qquad&  \forall \psi\in G_\mu^* \subset H^{p+1}(\Omega).
\end{alignedat}
\end{equation} 

Now we present the following error estimates.

\begin{theorem}[Error estimate on eigenvalues] \label{thm:ev}
Let $\mu \in \sigma(T)\setminus\{0\}$ with ascent $\varsigma$ and algebraic multiplicity $\varrho$ and $\mu_{h,j}$ be approximated eigenvalues, there holds 
\begin{equation} 
\max_{1 \le j \le\varrho } |\mu - \mu_{h,j}|^\varsigma  + |\mu - \langle \mu_h \rangle | \le C h^{2p},
\end{equation}
where $C$ is a constant, depending on $\mu$ (and on the mesh regularity, the polynomial degree $p$ and the domain $\Omega$) but independent of the mesh-size $h$.
\end{theorem}

\begin{proof}
We estimate each term on the right-hand side of \eqref{eq:cv_eigenval}. Firstly, for $\phi \in G_\mu $ and $\psi \in L$, the Cauchy-Schwarz inequality gives
\begin{align}
((T-T_h)(\phi),\psi)_{L}  \le  \| T(\phi) - T_h(\phi) \|_{L} \| \psi \|_{L}, 
\end{align}
where the smoothness assumption \eqref{eq:smT} and Theorem \ref{thm:ef} yields the estimate
\begin{align} \label{eq:e1}
((T-T_h)(\phi),\psi)_{L}  \le C h^{p+1} \| \phi \|_L \| \psi \|_L.
\end{align}
Similarly, for $\phi \in L $ and $\psi\in G_\mu^*$, one obtains
\begin{align} \label{eq:e2}
(\phi, (T^* - T^*_h)(\psi))_{L}  \le C h^{p+1} \| \phi \|_L \| \psi \|_L.
\end{align}

Secondly, for $\phi \in G_\mu $ and $\psi\in G_\mu^*$, using definitions \eqref{eq:Th} and \eqref{eq:T*h}, Galerkin orthogonality, and boundedness \eqref{eq:a}, we calculate
\begin{align}
((T-T_h)(\phi),\psi)_{L}  
 & = b(T(\phi),\psi) - b(T_h(\phi),\psi) \nonumber\\
&= a(T(\phi), T^*(\psi) ) - a(T_h(\phi), T^*_h (\psi )) \nonumber\\
& = a(T(\phi) - T_h(\phi), T^*(\psi) - T^*_h (\psi ) ) \nonumber\\
& \le \beta | T(\phi) - T_h(\phi) |_{H^1(\Omega)} | T^*(\psi) - T^*_h (\psi ) |_{H^1(\Omega)}, \label{eq:calcul1}
\end{align}
from which the smoothness assumption \eqref{eq:smT} and Theorem \ref{thm:ef} yield the estimate
\begin{align} \label{eq:e3}
((T-T_h)(\phi),\psi)_{L}  \le C h^{2p} \| \phi \|_L \| \psi \|_L.
\end{align}

Now combining all the estimates \eqref{eq:e1}, \eqref{eq:e2}, and \eqref{eq:e3} with Theorem \ref{thm:eve0} completes the proof.
\end{proof}

\begin{remark}[Error estimate on eigenvalues]
Since the eigenvalues $\lambda$ and $\lambda_h$ associated with~\eqref{eq:vf} and~\eqref{eq:vfh}, respectively, are such that $\lambda=\mu^{-1}$ and $\lambda_h=\mu_h^{-1}$, we infer that the same estimate as in Theorem \ref{thm:ev} holds true for the error between $\lambda$ and $\lambda_h$.
\end{remark}

Now we present a preliminary result on the eigenfunction error estimate.
\begin{lemma}[Eigenfunction error estimate in $L^2$]  \label{lem:ef}
Let $\mu \in \sigma(T)\setminus\{0\}$ with ascent $\varsigma$ and algebraic multiplicity $\varrho$. Let $u_{h,j}\in V^h$ be a unit vector in $\text{\em ker}(\mu_{h,j} I- T_h)^\ell$ for some positive integer $\ell \le \varsigma$. Then, for any integer $\varrho$ with $\ell \le \varrho \le \varsigma$, there is a unit vector $u_\varrho \in \text{\em ker}(\mu I-T)^\varrho \subset G_\mu$ such that
\begin{equation} 
\| u_\varrho - u_{h,j} \|_L \le C h^{p \frac{m-\ell+1}{\varsigma}},
\end{equation}
where $C$ is a constant depending on $\mu$ but independent of $h$.
\end{lemma}
\begin{proof}
Combining the estimate \eqref{eq:e1} with Theorem \ref{thm:efe0} completes the proof.
\end{proof}

\begin{theorem}[Eigenfunction error estimate in $H^1$] \label{thm:efh1}
Assume that $\varsigma=m=1$ for the exact eigenvalue $\mu$.
Since $\mu_h$ is simple, we drop the index $j$ for the 
approximate eigenfunction $u_h$. Assuming $u \in G_\mu \subset H^{p+1}(\Omega)$, there holds
\begin{equation} 
|u - u_h |_{H^1(\Omega)} \le Ch^p, 
\end{equation}
where $C$ is a constant depending on $\mu$ but independent of $h$.
\end{theorem}

\begin{proof}
The coercivity \eqref{eq:a} and the Pythagorean eigenvalue error identity (see \cite{strang1973analysis,hughes2014finite,puzyrev2017dispersion}) gives
\begin{equation}
\alpha |u - u_h |^2_{H^1(\Omega)} \le a(u - u_h, u - u_h) = \lambda \| u - u_h \|^2_{L^2(\Omega)} + \lambda_h - \lambda,
\end{equation}
which completes the proof by using Theorem \ref{thm:ev} and Lemma \ref{lem:ef}.
\end{proof}

\section{Numerical experiments} \label{sec:num}
In this section, we present the numerical validation of the standard FEM and SGFEM applied to both the source problem and the eigenvalue problem with an interface. If the interface is at one of the nodes in a mesh configuration, that is, a fitting mesh, then no enrichment is needed, hence one uses the standard FEM \cite{kergrene2016stable}. SGFEM is robust for non-fitting meshes. In the following numerical experiments, we utilize non-fitting meshes and discretize the domain uniformly. Let $\rho_p$ denote the convergence rate for $p$-th order approximation and we study the convergence behaviors as follows.

\subsection{Source problem}
We consider the source problem \eqref{eq:pdef} with the manufactured solution
\begin{equation} \label{eq:ex1u}
u(x) = 
\begin{cases}
\sin(6\pi x), & 0 \le x \le \gamma, \\
\frac{1}{2} \sin(3 \pi (x-1) ), & \gamma \le x \le 1,\\
\end{cases}
\end{equation}
where $\gamma = 1/3, \kappa_0(x) = 1, \kappa_1(x) = 4$ and $f(x)$ is the source function satisfying \eqref{eq:pdef}. This solution is $C^0$ but not $C^1$ at the interface $x=\gamma$. Thus, for non-fitting meshes, it is expected that the standard FEM solutions converge at a rate of $h^{1/2}$ in $H^1$ semi-norm for arbitrary order $p$.

We perform a uniform discretization of the domain into $N = 10,20,\cdots, 160$ elements. With these mesh configurations, the interface is alway located inside an element, hence they are non-fitting meshes.  Table \ref{tab:sp} shows the $H^1$ semi-norm errors of FEM and SGFEM solutions. For $p=1,2,3$, the $H^1$ semi-norm errors of the standard FEM solutions converge at an approximate rate of $h^{1/2}$ while those of SGFEM solutions converge optimally at rates of $h^p$, which validates Theorem \ref{thm:ef}.

\begin{table}[h!]
\centering 
\begin{tabular}{| c | c c | c  c| c  c| c |}
\hline
 & \multicolumn{2}{c|}{$p=1$} & \multicolumn{2}{c|}{$p=2$} & \multicolumn{2}{c|}{$p=3$}  \\
$N$ & FEM &  SGFEM & FEM &  SGFEM & FEM &  SGFEM  \\[0.1cm] \hline

10&	4.64E+0&	4.08E+0&	1.61E+0&	8.50E-1&	1.10E+0&	1.61E-1 \\[0.1cm] 
20&	2.67E+0&	2.09E+0&	1.02E+0&	2.29E-1&	8.35E-1&	2.03E-2\\[0.1cm] 
40&	1.57E+0&	1.06E+0&	7.41E-1&	6.22E-2&	5.15E-1&	2.67E-3\\[0.1cm] 
80&	1.05E+0&	5.31E-1&	4.59E-1&	1.56E-2&	4.12E-1&	3.39E-4\\[0.1cm] 
160&	6.40E-1&	2.66E-1&	3.72E-1&	3.97E-3&	2.57E-1&	4.92E-5\\[0.1cm] 
$\rho_p$	&	0.71&	0.99&	0.54&	1.94&	0.52&	2.93 \\[0.1cm] \hline

\end{tabular}
\caption{$H^1$ semi-norm errors of FEM and SGFEM solutions for $p=1,2,3$.}
\label{tab:sp} 
\end{table}

\subsection{Eigenvalue problems}
In this section, we consider the eigenvalue problem \eqref{eq:pde1d} with two cases where the exact eigenpairs are given in subsections \ref{sec:c2} and \ref{sec:c3}, that is, we consider the following two examples.

\begin{figure}[h!]
\centering
\includegraphics[height=6.55cm]{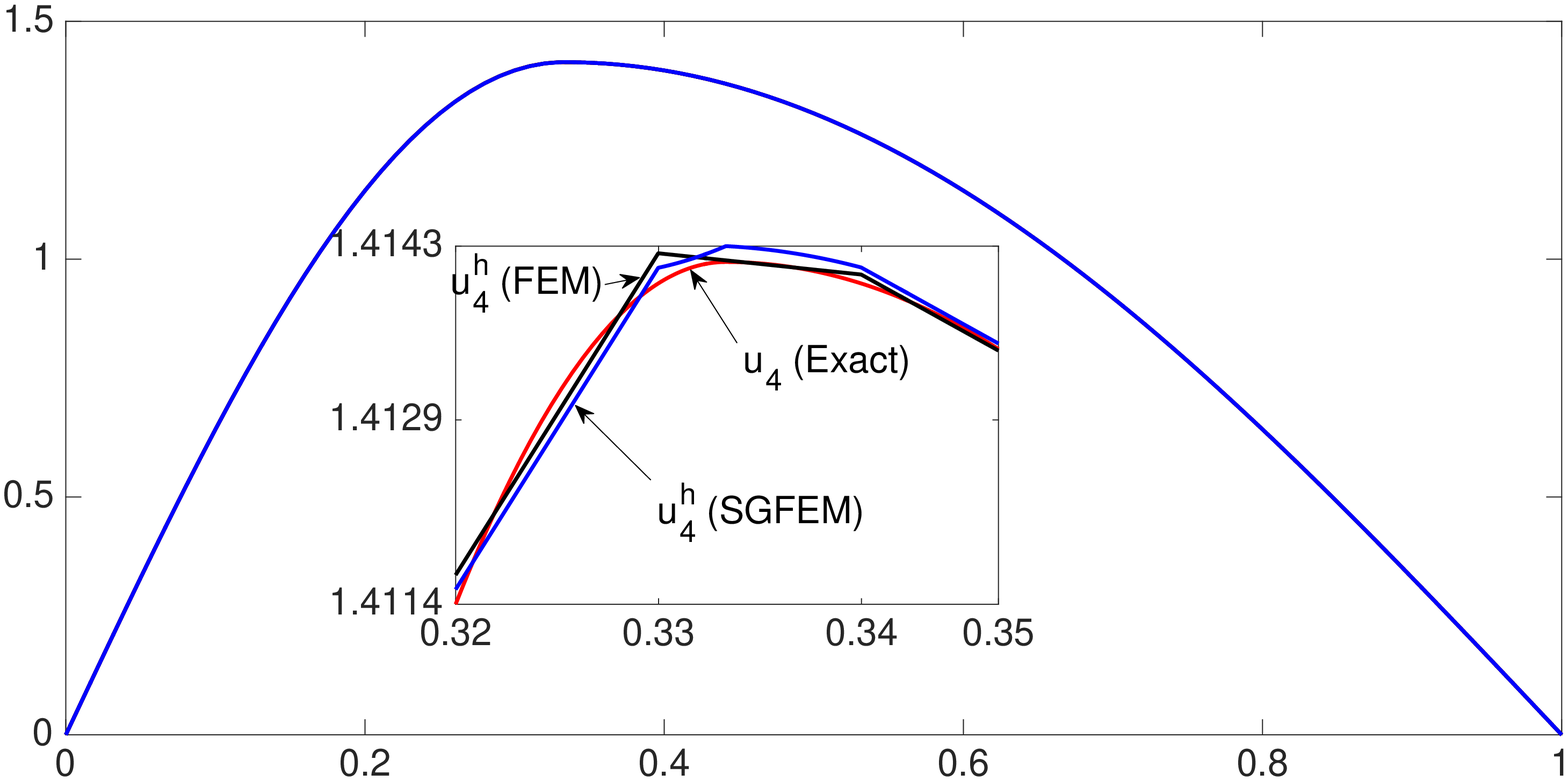} 
\caption{A comparison of eigenfunction $u_1$ approximated by FEM and SGFEM using $N=100$ linear elements for Example 1.}
\label{fig:efsex1u1}
\includegraphics[height=8.0cm]{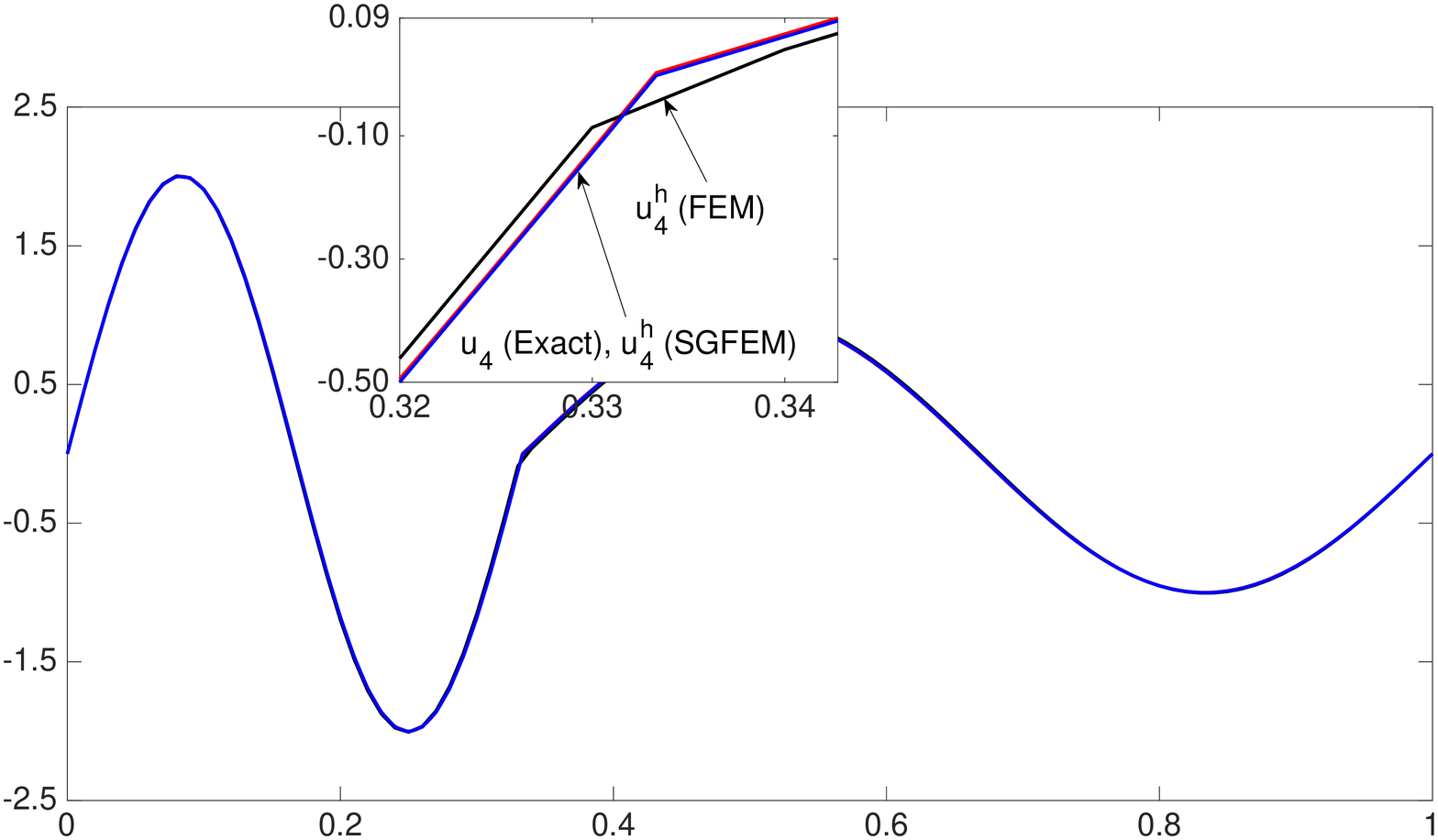} 
\caption{A comparison of eigenfunction $u_4$ approximated by FEM and SGFEM using $N=100$ linear elements for Example 1.}
\label{fig:efsex1u4}
\end{figure}

\begin{itemize}
\item \textbf{Example 1}: Eigenvalue problem \eqref{eq:pde1d} with $\gamma = 1/3$ and $\eta = 4$ as in subsection \ref{sec:c2}.
We choose this case because that (1) the exact eigenpairs are known, (2) the eigenfunction $u_j$ is $C^1$ continuous for $j=1,3,\cdots$ while only $C^0$ for $j=2,4,\cdots.$
\item \textbf{Example 2}: Eigenvalue problem \eqref{eq:pde1d} with $\gamma = 1/\pi$ and $\eta = e^2$ as in subsection \ref{sec:c3}. We choose this case as both $\gamma$ and $\eta$ are transcendental numbers and $x=\gamma$ is never a node on a uniform mesh. All the eigenfunctions are only $C^0$ at the interface.
\end{itemize}

We discretize the domain $\Omega = [0, 1]$ into $N=100$ uniform elements. Figure \ref{fig:efsex1u1} shows the comparison of the first eigenfunction $u_1$ for Example 1 when approximated by the standard FEM and SGFEM with linear elements, while Figure \ref{fig:efsex1u4} shows that of the fourth eigenfunction $u_4$. Since $u_1$ is $C^1$ at the interface, Figure \ref{fig:efsex1u1} shows that both standard FEM and SGFEM approximate $u_1$ equally well when using linear elements. SGFEM approximates $u_1$ slightly better when looking into the eigenfunction errors in both $H^1$ semi-norm (see Table \ref{tab:ex1efh1}) and $L^2$ norm (see Table \ref{tab:ex1efl2}). We will discuss this in more detail when we present the tables of errors. Since $u_4$ is only $C^0$ at the interface, Figure \ref{fig:efsex1u4} shows that SGFEM outperforms the standard FEM significantly.

Now we collect the relative eigenvalue errors, which is defined as
$
\frac{|\lambda^h - \lambda| }{\lambda},
$
where $\lambda$ is the exact eigenvalue. Table \ref{tab:ex1ev} shows the relative eigenvalue errors of $\lambda_1, \lambda_4$, and $\lambda_8$ when using both FEM and SGFEM with linear, quadratic, and cubic elements for Example 1, while Table \ref{tab:ex2ev} shows those of Example 2. Both tables show that SGFEM outperforms FEM consistently. When using FEM, the eigenvalue errors converge at an optimal rate only when using linear elements for $C^1$ eigenfunctions, that is $\lambda_1$ of Example 1 (see Table \ref{tab:ex1ev}). For all the scenarios, the eigenvalue errors when using SGFEM converge at optimal rates, that is $h^{2p}$. This validates the theoretical eigenvalue error estimate in Theorem \ref{thm:ev}.

\begin{table}[h!]
\centering 
\begin{tabular}{| c | c | c c | c  c| c  c| c |}
\hline
&  & \multicolumn{2}{c|}{$\lambda_1$} & \multicolumn{2}{c|}{$\lambda_4$} & \multicolumn{2}{c|}{$\lambda_8$}  \\
$p$ & $N$ & FEM &  SGFEM & FEM &  SGFEM & FEM &  SGFEM  \\[0.1cm] \hline

	&10&	8.66E-3&	7.61E-3&	2.78E-1&	2.22E-1&	1.11E+0&	7.63E-1 \\[0.1cm] 
	&20&	2.46E-3&	2.08E-3&	1.01E-1&	5.53E-2&	2.07E-1&	1.95E-1 \\[0.1cm] 
1	&40&	5.69E-4&	5.53E-4&	3.10E-2&	1.39E-2&	7.22E-2&	5.59E-2 \\[0.1cm] 
	&80&	1.47E-4&	1.41E-4&	1.60E-2&	3.47E-3&	2.62E-2&	1.39E-2 \\[0.1cm] 
	&160&	3.60E-5&	3.58E-5&	5.05E-3&	8.68E-4&	7.67E-3&	3.47E-3 \\[0.1cm] 
	&$\rho_1$	&	1.99&	1.94&	1.42&	2.00&	1.73&	1.94 \\[0.1cm]  \hline
	
	&10&	1.98E-4&	2.55E-5&	4.17E-2&	7.70E-3&	1.33E-1&	1.10E-1 \\[0.1cm] 
	&20&	2.70E-5&	1.60E-6&	1.43E-2&	5.98E-4&	2.69E-2&	8.85E-3 \\[0.1cm] 
2	&40&	2.81E-6&	1.00E-7&	9.66E-3&	4.42E-5&	1.00E-2&	6.84E-4 \\[0.1cm] 
	&80&	4.08E-7&	6.27E-9&	3.15E-3&	2.80E-6&	3.24E-3&	4.45E-5 \\[0.1cm] 
	&160&	4.27E-8&	3.93E-10&	2.42E-3&	1.82E-7&	2.41E-3&	2.90E-6 \\[0.1cm] 
	&$\rho_2$	&	3.04&	4.00&	1.04&	3.85&	1.46&	3.81 \\[0.1cm]  \hline

	&10&	4.09E-5&	2.96E-8&	1.82E-2&	2.76E-4&	2.93E-2&	9.20E-3 \\[0.1cm] 
	&20&	3.23E-6&	4.70E-10&	1.26E-2&	4.47E-6&	1.40E-2&	2.50E-4 \\[0.1cm] 
3	&40&	6.46E-7&	8.50E-12&	4.33E-3&	7.25E-8&	4.35E-3&	4.54E-6 \\[0.1cm] 
	&80&	4.99E-8&	4.13E-13&	3.02E-3&	1.14E-9&	3.05E-3&	7.24E-8 \\[0.1cm] 
	&$\rho_3$	&	3.13&	5.40&	0.93&	5.96&	1.15&	5.66 \\[0.1cm]  \hline

\end{tabular}
\caption{Relative eigenvalue errors for Example 1 using both FEM and SGFEM with $p=1,2,3$ for the first, fourth, and eighth eigenvalues.}
\label{tab:ex1ev} 
\end{table}

\begin{table}[h!]
\centering 
\begin{tabular}{| c | c | c c | c  c| c  c| c |}
\hline
&  & \multicolumn{2}{c|}{$\lambda_1$} & \multicolumn{2}{c|}{$\lambda_4$} & \multicolumn{2}{c|}{$\lambda_8$}  \\
$p$ & $N$ & FEM &  SGFEM & FEM &  SGFEM & FEM &  SGFEM  \\[0.1cm] \hline

	&10&	1.27E-2&	1.15E-2&	3.36E-1&	2.43E-1&	1.89E+0&	1.29E+0 \\[0.1cm] 
	&20&	4.15E-3&	3.03E-3&	1.21E-1&	5.91E-2&	2.44E-1&	1.70E-1 \\[0.1cm] 
1	&40&	1.79E-3&	7.86E-4&	5.41E-2&	1.45E-2&	5.87E-2&	3.96E-2 \\[0.1cm] 
	&80&	7.54E-4&	2.02E-4&	1.64E-2&	3.62E-3&	1.13E-2&	1.00E-2 \\[0.1cm] 
	&160&	2.78E-4&	5.08E-5&	5.86E-3&	9.04E-4&	2.85E-3&	2.54E-3 \\[0.1cm] 
	&$\rho_1$	&	1.35&	1.95&	1.46&	2.02&	2.32&	2.21 \\[0.1cm]  \hline
				
	&10&	1.38E-3&	4.66E-5&	5.98E-2&	1.32E-2&	1.71E-1&	1.18E-1 \\[0.1cm] 
	&20&	9.81E-4&	2.96E-6&	3.03E-2&	9.64E-4&	1.60E-2&	8.73E-3 \\[0.1cm] 
2	&40&	4.25E-4&	1.86E-7&	7.68E-3&	6.31E-5&	1.03E-3&	6.31E-4 \\[0.1cm] 
	&80&	2.93E-4&	1.17E-8&	6.01E-3&	4.26E-6&	2.85E-4&	4.22E-5 \\[0.1cm] 
	&160&	1.17E-4&	7.41E-10&	2.04E-3&	2.67E-7&	2.14E-5&	2.66E-6 \\[0.1cm] 
	&$\rho_2$	&	0.88&	3.97&	1.21&	3.90&	3.17&	3.86 \\[0.1cm]  \hline

	&10&	9.97E-4&	1.35E-7&	3.01E-2&	4.21E-4&	2.06E-2&	1.24E-2 \\[0.1cm] 
	&20&	5.51E-4&	2.12E-9&	1.19E-2&	7.55E-6&	1.36E-3&	2.62E-4 \\[0.1cm] 
3	&40&	3.60E-4&	3.55E-11&		7.16E-3&	1.22E-7&	2.70E-4&	4.92E-6 \\[0.1cm] 
	&80&	1.69E-4&	9.67E-14&	3.04E-3&	1.93E-9&	4.86E-5&	8.42E-8 \\[0.1cm] 
	&$\rho_3$	&	0.83&	6.71&	1.07&	5.92&	2.85&	5.73 \\[0.1cm]  \hline

\end{tabular}
\caption{Relative eigenvalue errors for Example 2 using both FEM and SGFEM with $p=1,2,3$ for the first, fourth, and eighth eigenvalues.}
\label{tab:ex2ev} 
\end{table}

Now we present the corresponding eigenfunction errors. We focus on Example 1 and collect both $H^1$ semi-norm and $L^2$ norm errors. Table \ref{tab:ex1efh1} shows the first, fourth, and eighth eigenfunction errors in $H^1$ semi-norm when using both FEM and SGFEM with linear, quadratic, and cubic elements, while Table \ref{tab:ex1efl2} shows the errors in $L^2$ norm. Similarly, both tables show that SGFEM leads to significantly smaller errors, hence outperforms FEM consistently. Since $u_1$ is $C^1$ at the interface, this $C^1$ regularity helps deliver optimal error convergence rates when using linear FEM. For all the scenarios, the eigenfunction errors in both $H^1$ semi-norm and $L^2$ norm when using SGFEM converge at optimal rates, that is $h^p$ and $h^{p+1}$, respectively. This validates the theoretical eigenfunction error estimate (in $H^1$ semi-norm) in Theorem \ref{thm:efh1}.

\begin{table}[h!]
\centering 
\begin{tabular}{| c | c | c c | c  c| c  c| c |}
\hline
&  & \multicolumn{2}{c|}{$u^h_1$} & \multicolumn{2}{c|}{$u^h_4$} & \multicolumn{2}{c|}{$u^h_8$}  \\
$p$ & $N$ & FEM &  SGFEM & FEM &  SGFEM & FEM &  SGFEM  \\[0.1cm] \hline

	&10&	3.81E-1&	3.59E-1&	9.70E+0&	8.90E+0&	4.66E+1&	4.64E+1 \\[0.1cm] 
	&20&	2.01E-1&	1.83E-1&	5.42E+0&	4.29E+0&	1.90E+1&	1.84E+1 \\[0.1cm] 
1	&40&	9.74E-2&	9.61E-2&	3.18E+0&	2.13E+0&	9.88E+0&	8.86E+0 \\[0.1cm] 
	&80&	4.94E-2&	4.82E-2&	2.13E+0&	1.06E+0&	5.66E+0&	4.30E+0 \\[0.1cm] 
	&160&	2.45E-2&	2.44E-2&	1.29E+0&	5.32E-1&	3.18E+0&	2.13E+0 \\[0.1cm] 
	&$\rho_1$	&	0.99&	0.97&	0.72&	1.01&	0.95&	1.10 \\[0.1cm]  \hline

	&10&	5.70E-2&	2.28E-2&	3.23E+0&	1.72E+0&	1.72E+1&	1.61E+1 \\[0.1cm] 
	&20&	1.77E-2&	5.71E-3&	2.07E+0&	4.58E-1&	6.32E+0&	3.69E+0 \\[0.1cm] 
2	&40&	6.68E-3&	1.43E-3&	1.50E+0&	1.24E-1&	3.22E+0&	9.85E-1 \\[0.1cm] 
	&80&	2.14E-3&	3.57E-4&	9.22E-1&	3.12E-2&	1.92E+0&	2.49E-1 \\[0.1cm] 
	&160&	8.19E-4&	8.93E-5&	7.46E-1&	7.94E-3&	1.51E+0&	6.35E-2 \\[0.1cm] 
	&$\rho_2$	&	1.53&	2.00&	0.54&	1.94&	0.87&	1.99 \\[0.1cm]  \hline

	&10&	1.99E-2&	8.02E-4&	2.23E+0&	3.19E-1&	6.52E+0&	3.99E+0 \\[0.1cm] 
	&20&	7.17E-3&	1.01E-4&	1.71E+0&	3.99E-2&	3.85E+0&	6.07E-1 \\[0.1cm] 
3	&40&	2.50E-3&	1.37E-5&	1.04E+0&	5.07E-3&	2.14E+0&	8.05E-2 \\[0.1cm] 
	&80&	8.89E-4&	1.71E-6&	8.28E-1&	6.34E-4&	1.69E+0&	1.01E-2 \\[0.1cm] 
	&160&	3.12E-4&	2.18E-7&	5.15E-1&	7.93E-5&	1.04E+0&	1.27E-3 \\[0.1cm] 
	&$\rho_3$	&	0.50&	2.96&	0.53&	2.99&	0.65&	2.91 \\[0.1cm]  \hline

\end{tabular}
\caption{Eigenfunction errors in $H^1$ semi-norm for Example 1 using both FEM and SGFEM with $p=1,2,3$ for the first, fourth, and eighth eigenfunctions.}
\label{tab:ex1efh1} 
\end{table}

\begin{table}[h!]
\centering 
\begin{tabular}{| c | c | c c | c  c| c  c| c |}
\hline
&  & \multicolumn{2}{c|}{$u^h_1$} & \multicolumn{2}{c|}{$u^h_4$} & \multicolumn{2}{c|}{$u^h_8$}  \\
$p$ & $N$ & FEM &  SGFEM & FEM &  SGFEM & FEM &  SGFEM  \\[0.1cm] \hline

	&10&	5.55E-3&	9.32E-3&	3.46E-1&	2.37E-1&	1.46E+0&	1.28E+0 \\[0.1cm] 
	&20&	1.79E-3&	2.27E-3&	1.29E-1&	5.78E-2&	4.26E-1&	3.67E-1 \\[0.1cm] 
1	&40&	3.52E-4&	6.02E-4&	3.62E-2&	1.39E-2&	1.35E-1&	9.86E-2 \\[0.1cm] 
	&80&	1.11E-4&	1.49E-4&	2.33E-2&	3.48E-3&	5.49E-2&	2.48E-2 \\[0.1cm] 
	&160&	2.21E-5&	3.77E-5&	7.74E-3&	8.69E-4&	1.67E-2&	6.20E-3 \\[0.1cm] 
	&$\rho_1$	&	2.00&	1.98&	1.34&	2.02&	1.59&	1.93 \\[0.1cm]  \hline

	&10&	1.47E-3&	3.52E-4&	6.52E-2&	2.98E-2&	4.38E-1&	3.79E-1 \\[0.1cm] 
	&20&	1.58E-4&	4.40E-5&	2.77E-2&	3.63E-3&	7.71E-2&	3.56E-2 \\[0.1cm] 
2	&40&	9.49E-5&	5.51E-6&	1.71E-2&	4.83E-4&	3.34E-2&	4.16E-3 \\[0.1cm] 
	&80&	4.86E-6&	6.89E-7&	5.92E-3&	6.02E-5&	1.19E-2&	4.92E-4 \\[0.1cm] 
	&160&	5.97E-6&	8.61E-8&	4.38E-3&	7.66E-6&	8.69E-3&	6.16E-5 \\[0.1cm] 
	&$\rho_2$	&	2.09&	3.00&	1.00&	2.98&	1.40&	3.15 \\[0.1cm]  \hline

	&10&	2.72E-4&	8.48E-6&	3.34E-2&	3.50E-3&	8.53E-2&	4.87E-2 \\[0.1cm] 
	&20&	8.94E-5&	5.31E-7&	2.35E-2&	2.13E-4&	5.07E-2&	3.40E-3 \\[0.1cm] 
3	&40&	1.44E-5&	3.61E-8&	7.93E-3&	1.34E-5&	1.58E-2&	2.14E-4 \\[0.1cm] 
	&80&	5.12E-6&	2.26E-9&	5.58E-3&	8.36E-7&	1.12E-2&	1.34E-5 \\[0.1cm] 
	&160&	8.50E-7&	1.44E-10&	1.98E-3&	5.23E-8&	3.92E-3&	8.36E-7 \\[0.1cm] 
	&$\rho_3$	&	2.08&	3.96&	1.02&	4.00&	1.11&	3.96 \\[0.1cm]  \hline

\end{tabular}
\caption{Eigenfunction errors in $L^2$ norm for Example 1 using both FEM and SGFEM with $p=1,2,3$ for the first, fourth, and eighth eigenfunctions.}
\label{tab:ex1efl2} 
\end{table}

\section{Concluding remarks} \label{sec:con}

In this paper, we focus on 1D and study the numerical approximation of eigenvalue problem with interfaces. We first generalize and develop arbitrary order SGFEMs for the interface source problem. Optimal error estimates in both $H^1$ semi-norm and $L^2$ norm are established. We then apply the abstract theory of the spectral approximation of compact operators to establish the eigenvalue and eigenfunction errors. 

Extension of this arbitrary order SGFEM to multiple dimensions is a challenging task. The key is to enrich the standard FEM space with enrichments which lead to a space containing a piecewise $p$-th order polynomial that interpolates the exact solution in an optimal fashion. In 1D, this is to show the existence of such enrichment satisfying \eqref{eq:ie} in Lemma \ref{lem:uh*}.  In multiple dimensions, to develop high order SGFEMs, it is challenging to develop the enrichments such that (1) satisfy this existence condition (2) the conditioning number of the resulting stiffness matrix is not significantly larger than that of standard FEM. Another direction of future work is the development of isogeometric elements for the interface problems. In general, for eigenvalue problem without an interface, isogeometric elements improve the spectral approximation significantly due to the high continuities of the isogeometric basis functions. However, these high continuities pose new challenges to develop enrichments for the problems with interfaces.

\section*{Acknowledgement}
This publication was made possible in part by the CSIRO Professorial Chair in Computational Geoscience at Curtin University and the Deep Earth Imaging Enterprise Future Science Platforms of the Commonwealth Scientific Industrial Research Organisation, CSIRO, of Australia. Additional support was provided by the European Union's Horizon 2020 Research and Innovation Program of the Marie Sk{\l}odowska-Curie grant agreement No. 777778, the Mega-grant of the Russian Federation Government (N 14.Y26.31.0013), the Institute for Geoscience Research (TIGeR), and the Curtin Institute for Computation. The J. Tinsley Oden Faculty Fellowship Research Program at the Institute for Computational Engineering and Sciences (ICES) of the University of Texas at Austin has partially supported the visits of VMC to ICES. The authors thank Professor I. Babu\v{s}ka (University of Texas at Austin) for stimulating and insightful discussions.



\bibliographystyle{elsarticle-harv}\biboptions{square,sort,comma,numbers}
\bibliography{ref}

\begin{thebibliography}{47}
\expandafter\ifx\csname natexlab\endcsname\relax\def\natexlab#1{#1}\fi
\expandafter\ifx\csname url\endcsname\relax
  \def\url#1{\texttt{#1}}\fi
\expandafter\ifx\csname urlprefix\endcsname\relax\def\urlprefix{URL }\fi

\bibitem[{Abdelaziz and Hamouine(2008)}]{abdelaziz2008survey}
Abdelaziz, Y., Hamouine, A., 2008. A survey of the extended finite element.
  Computers \& structures 86~(11-12), 1141--1151.

\bibitem[{Adams and Fournier(2003)}]{adams2003sobolev}
Adams, R.~A., Fournier, J.~J., 2003. Sobolev spaces. Vol. 140. Academic press.

\bibitem[{Ainsworth and Wajid(2010)}]{ainsworth2010optimally}
Ainsworth, M., Wajid, H.~A., 2010. Optimally blended spectral-finite element
  scheme for wave propagation and nonstandard reduced integration. SIAM Journal
  on Numerical Analysis 48~(1), 346--371.

\bibitem[{Antonietti et~al.(2006)Antonietti, Buffa, and
  Perugia}]{antonietti2006discontinuous}
Antonietti, P.~F., Buffa, A., Perugia, I., 2006. Discontinuous {G}alerkin
  approximation of the {L}aplace eigenproblem. Comput. Methods Appl. Mech.
  Engrg. 195~(25), 3483--3503.

\bibitem[{Babu{\v{s}}ka and Banerjee(2012)}]{babuvska2012stable}
Babu{\v{s}}ka, I., Banerjee, U., 2012. Stable generalized finite element method
  ({SGFEM}). Computer Methods in Applied Mechanics and Engineering 201,
  91--111.

\bibitem[{Babu{\v{s}}ka et~al.(2017)Babu{\v{s}}ka, Banerjee, and
  Kergrene}]{babuvska2017strongly}
Babu{\v{s}}ka, I., Banerjee, U., Kergrene, K., 2017. Strongly stable
  generalized finite element method: Application to interface problems.
  Computer Methods in Applied Mechanics and Engineering 327, 58--92.

\bibitem[{Babuska and Melenk(1995)}]{babuska1995partition}
Babuska, I., Melenk, J.~M., 1995. The partition of unity finite element method.
  Tech. rep., Physical science and technology, Maryland University.

\bibitem[{Babu{\v{s}}ka and Osborn(1991)}]{babuvska1991eigenvalue}
Babu{\v{s}}ka, I., Osborn, J., 1991. Eigenvalue problems. In: Handbook of
  Numerical Analysis, {V}ol.\ {II}. Handb. Numer. Anal., II. North-Holland,
  Amsterdam, pp. 641--787.

\bibitem[{Barto{\v{n}} et~al.(2018)Barto{\v{n}}, Calo, Deng, and
  Puzyrev}]{bartovn2018generalization}
Barto{\v{n}}, M., Calo, V., Deng, Q., Puzyrev, V., 2018. Generalization of the
  {P}ythagorean eigenvalue error theorem and its application to isogeometric
  analysis. In: Numerical Methods for PDEs. Springer, pp. 147--170.

\bibitem[{Belytschko et~al.(2009)Belytschko, Gracie, and
  Ventura}]{belytschko2009review}
Belytschko, T., Gracie, R., Ventura, G., 2009. A review of extended/generalized
  finite element methods for material modeling. Modelling and Simulation in
  Materials Science and Engineering 17~(4), 043001.

\bibitem[{Bramble and Osborn(1973)}]{bramble1973rate}
Bramble, J.~H., Osborn, J.~E., 1973. Rate of convergence estimates for
  nonselfadjoint eigenvalue approximations. Math. Comp. 27~(123), 525--549.

\bibitem[{Brenner and Scott(2008)}]{brenner2008mathematical}
Brenner, S.~C., Scott, L.~R., 2008. The mathematical theory of finite element
  methods, 3rd Edition. Vol.~15 of Texts in Applied Mathematics. Springer, New
  York.
\newline\urlprefix\url{http://dx.doi.org/10.1007/978-0-387-75934-0}

\bibitem[{Calo et~al.(2017{\natexlab{a}})Calo, Deng, and
  Puzyrev}]{calo2017quadrature}
Calo, V., Deng, Q., Puzyrev, V., 2017{\natexlab{a}}. Quadrature blending for
  isogeometric analysis. Procedia Computer Science 108, 798--807.

\bibitem[{Calo et~al.(To appear)Calo, Cicuttin, Deng, and
  Ern}]{calo2017spectral}
Calo, V.~M., Cicuttin, M., Deng, Q., Ern, A., To appear. Spectral approximation
  of elliptic operators by the {H}ybrid {H}igh--{O}rder method. Mathematics of
  Computation.

\bibitem[{Calo et~al.(2017{\natexlab{b}})Calo, Deng, and
  Puzyrev}]{calo2017dispersion}
Calo, V.~M., Deng, Q., Puzyrev, V., 2017{\natexlab{b}}. Dispersion optimized
  quadratures for isogeometric analysis. arXiv preprint arXiv:1702.04540.

\bibitem[{Canuto(1978)}]{canuto1978eigenvalue}
Canuto, C., 1978. Eigenvalue approximations by mixed methods. RAIRO Anal.
  Num{\'e}r. 12~(1), 27--50.

\bibitem[{Chatelin(1975)}]{chatelin1983spectral}
Chatelin, F., 1975. Spectral approximation of linear operators. 1983.

\bibitem[{Ciarlet(2002)}]{ciarlet2002finite}
Ciarlet, P.~G., 2002. Finite Element Method for Elliptic Problems. Society for
  Industrial and Applied Mathematics, Philadelphia, PA, USA.

\bibitem[{Cottrell et~al.(2006)Cottrell, Reali, Bazilevs, and
  Hughes}]{cottrell2006isogeometric}
Cottrell, J.~A., Reali, A., Bazilevs, Y., Hughes, T. J.~R., 2006. Isogeometric
  analysis of structural vibrations. Computer methods in applied mechanics and
  engineering 195~(41), 5257--5296.

\bibitem[{Deng et~al.(2018{\natexlab{a}})Deng, Barto{\v{n}}, Puzyrev, and
  Calo}]{deng2018dispersion}
Deng, Q., Barto{\v{n}}, M., Puzyrev, V., Calo, V.~M., 2018{\natexlab{a}}.
  Dispersion-minimizing quadrature rules for {C}1 quadratic isogeometric
  analysis. Computer Methods in Applied Mechanics and Engineering 328,
  554--564.

\bibitem[{Deng and Calo(2018)}]{deng2017dispersion}
Deng, Q., Calo, V., 2018. Dispersion-minimized mass for isogeometric analysis.
  Computer Methods in Applied Mechanics and Engineering 341, 71--92.

\bibitem[{Deng et~al.(2018{\natexlab{b}})Deng, Puzyrev, and
  Calo}]{deng2018isogeometric}
Deng, Q., Puzyrev, V., Calo, V., 2018{\natexlab{b}}. Isogeometric spectral
  approximation for elliptic differential operators. Journal of Computational
  Science.

\bibitem[{Deng et~al.(2019)Deng, Puzyrev, and Calo}]{deng2019optimal}
Deng, Q., Puzyrev, V., Calo, V., 2019. Optimal spectral approximation of
  2n-order differential operators by mixed isogeometric analysis. Computer
  Methods in Applied Mechanics and Engineering 343, 297--313.

\bibitem[{Ern and Guermond(2004)}]{ern2004theory}
Ern, A., Guermond, J.-L., 2004. Theory and practice of finite elements. Vol.
  159 of Applied Mathematical Sciences. Springer-Verlag, New York.

\bibitem[{Fries and Belytschko(2010)}]{fries2010extended}
Fries, T.-P., Belytschko, T., 2010. The extended/generalized finite element
  method: an overview of the method and its applications. International Journal
  for Numerical Methods in Engineering 84~(3), 253--304.

\bibitem[{Giani(2015)}]{giani2015hp}
Giani, S., 2015. hp-adaptive composite discontinuous {G}alerkin methods for
  elliptic eigenvalue problems on complicated domains. Appl. Math. Comput. 267,
  604--617.

\bibitem[{Gopalakrishnan et~al.(2015)Gopalakrishnan, Li, Nguyen, and
  Peraire}]{gopalakrishnan2015spectral}
Gopalakrishnan, J., Li, F., Nguyen, N.-C., Peraire, J., 2015. Spectral
  approximations by the {HDG} method. Math. Comp. 84~(293), 1037--1059.

\bibitem[{Grisvard(1985)}]{grisvard1985elliptic}
Grisvard, P., 1985. Elliptic problems in nonsmooth domains. Vol.~24 of
  Monographs and Studies in Mathematics. Pitman (Advanced Publishing Program),
  Boston, MA.

\bibitem[{Hughes et~al.(2014)Hughes, Evans, and Reali}]{hughes2014finite}
Hughes, T. J.~R., Evans, J.~A., Reali, A., 2014. Finite element and {NURBS}
  approximations of eigenvalue, boundary-value, and initial-value problems.
  Computer Methods in Applied Mechanics and Engineering 272, 290--320.

\bibitem[{Hughes et~al.(2008)Hughes, Reali, and Sangalli}]{hughes2008duality}
Hughes, T. J.~R., Reali, A., Sangalli, G., 2008. Duality and unified analysis
  of discrete approximations in structural dynamics and wave propagation:
  comparison of p-method finite elements with k-method {NURBS}. Computer
  methods in applied mechanics and engineering 197~(49), 4104--4124.

\bibitem[{Johnson(2009)}]{johnson2009numerical}
Johnson, C., 2009. Numerical solution of partial differential equations by the
  finite element method. Dover Publications, Inc., Mineola, NY, reprint of the
  1987 edition.

\bibitem[{Kergrene et~al.(2016)Kergrene, Babu{\v{s}}ka, and
  Banerjee}]{kergrene2016stable}
Kergrene, K., Babu{\v{s}}ka, I., Banerjee, U., 2016. Stable generalized finite
  element method and associated iterative schemes; application to interface
  problems. Computer Methods in Applied Mechanics and Engineering 305, 1--36.

\bibitem[{Laborde et~al.(2005)Laborde, Pommier, Renard, and
  Sala{\"u}n}]{laborde2005high}
Laborde, P., Pommier, J., Renard, Y., Sala{\"u}n, M., 2005. High-order extended
  finite element method for cracked domains. International Journal for
  Numerical Methods in Engineering 64~(3), 354--381.

\bibitem[{Melenk(1995)}]{melenk1995generalized}
Melenk, J.~M., 1995. On generalized finite element methods. Ph.D. thesis,
  research directed by Dept. of Mathematics.University of Maryland at College
  Park.

\bibitem[{Melenk and Babu{\v{s}}ka(1996)}]{melenk1996partition}
Melenk, J.~M., Babu{\v{s}}ka, I., 1996. The partition of unity finite element
  method: basic theory and applications. Computer methods in applied mechanics
  and engineering 139~(1-4), 289--314.

\bibitem[{Mercier et~al.(1981)Mercier, Osborn, Rappaz, and
  Raviart}]{mercier1981eigenvalue}
Mercier, B., Osborn, J.~E., Rappaz, J., Raviart, P.-A., 1981. Eigenvalue
  approximation by mixed and hybrid methods. Math. Comp. 36~(154), 427--453.

\bibitem[{Mercier and Rappaz(1978)}]{mercier1978eigenvalue}
Mercier, B., Rappaz, J., 1978. Eigenvalue approximation via non-conforming and
  hybrid finite element methods. Publications des s{\'e}minaires de
  math{\'e}matiques et informatique de Rennes 1978~(S4), 1--16.

\bibitem[{Osborn(1975)}]{osborn1975spectral}
Osborn, J.~E., 1975. Spectral approximation for compact operators. Math. Comp.
  29~(131), 712--725.

\bibitem[{Puzyrev et~al.(2018)Puzyrev, Deng, and Calo}]{puzyrev2018spectral}
Puzyrev, V., Deng, Q., Calo, V., 2018. Spectral approximation properties of
  isogeometric analysis with variable continuity. Computer Methods in Applied
  Mechanics and Engineering 334, 22--39.

\bibitem[{Puzyrev et~al.(2017)Puzyrev, Deng, and Calo}]{puzyrev2017dispersion}
Puzyrev, V., Deng, Q., Calo, V.~M., 2017. Dispersion-optimized quadrature rules
  for isogeometric analysis: modified inner products, their dispersion
  properties, and optimally blended schemes. Computer Methods in Applied
  Mechanics and Engineering 320, 421--443.

\bibitem[{Stazi et~al.(2003)Stazi, Budyn, Chessa, and
  Belytschko}]{stazi2003extended}
Stazi, F., Budyn, E., Chessa, J., Belytschko, T., 2003. An extended finite
  element method with higher-order elements for curved cracks. Computational
  Mechanics 31~(1-2), 38--48.

\bibitem[{Strang and Fix(1973)}]{strang1973analysis}
Strang, G., Fix, G.~J., 1973. An analysis of the finite element method. Vol.
  212. Prentice-hall Englewood Cliffs, NJ.

\bibitem[{Strouboulis et~al.(2000{\natexlab{a}})Strouboulis, Babu{\v{s}}ka, and
  Copps}]{strouboulis2000design}
Strouboulis, T., Babu{\v{s}}ka, I., Copps, K., 2000{\natexlab{a}}. The design
  and analysis of the generalized finite element method. Computer methods in
  applied mechanics and engineering 181~(1-3), 43--69.

\bibitem[{Strouboulis et~al.(2000{\natexlab{b}})Strouboulis, Copps, and
  Babu{\v{s}}ka}]{strouboulis2000generalized}
Strouboulis, T., Copps, K., Babu{\v{s}}ka, I., 2000{\natexlab{b}}. The
  generalized finite element method: an example of its implementation and
  illustration of its performance. International Journal for Numerical Methods
  in Engineering 47~(8), 1401--1417.

\bibitem[{Strouboulis et~al.(2001)Strouboulis, Copps, and
  Babu{\v{s}}ka}]{strouboulis2001generalized}
Strouboulis, T., Copps, K., Babu{\v{s}}ka, I., 2001. The generalized finite
  element method. Computer methods in applied mechanics and engineering
  190~(32-33), 4081--4193.

\bibitem[{S{\"u}li(2012)}]{suli2012lecture}
S{\"u}li, E., 2012. Lecture notes on finite element methods for partial
  differential equations. Mathematical Institute, University of Oxford.

\bibitem[{Zhang et~al.(2014)Zhang, Banerjee, and
  Babu{\v{s}}ka}]{zhang2014higher}
Zhang, Q., Banerjee, U., Babu{\v{s}}ka, I., 2014. Higher order stable
  generalized finite element method. Numerische Mathematik 128~(1), 1--29.

\end{thebibliography}

\end{document}